\documentclass[10pt,a4paper]{article}
\usepackage[utf8]{inputenc}
\usepackage{amsmath}
\usepackage{cite}
\usepackage{amsfonts}
\usepackage{amssymb}
\usepackage{tabularx,array,amsmath}
\usepackage{enumerate,graphicx}
\usepackage{amsfonts, amsthm}
\usepackage{amssymb}
\usepackage{multirow}
\usepackage{bigstrut}
\usepackage{hyperref}

\newtheorem{theorem}{Theorem}[section]
\newtheorem{lemma}[theorem]{Lemma}
\newtheorem{proposition}[theorem]{Proposition}
\newtheorem{corollary}[theorem]{Corollary}

\theoremstyle{definition}
\newtheorem{definition}[theorem]{Definition}

\newcommand{\V}{\mathcal{V}}

\newcommand{\T}{\mathcal{T}}

\newcommand{\modd}{\ \mbox{mod} \ }

\begin{document}

\title{Representation zeta function of a family of maximal class groups: various exceptional cases}
\author{Shannon Ezzat}
\maketitle
\begin{center}
\emph{Address:} Cape Breton University\\
P.O.Box 5300, 1250 Grand Lake Road,\\
Sydney, Nova Scotia, Canada B1P 6L2\\
\emph{Email:} \href{mailto:shannon_ezzat@cbu.ca}{shannon\_ezzat@cbu.ca}
\end{center}
\begin{abstract}
This paper is a sequel to \emph{Representation zeta function of a family of maximal class groups: Non-exceptional primes} \cite{MaxClassI}. We use an explicitly constructive method developed in the prequel paper to calculate some exceptional cases of $p$-local representation zeta functions of a family of finitely generated nilpotent groups $M_n$ with maximal nilpotency class. We construct all irreducible representations of degree $p^N$ for all $N\in \mathbb{N}$ for the group $M_{p+1}$ for a fixed prime $p$. We also construct all irreducible representations of degree $2^N$ for the group $M_4$. Together with the main result from the prequel, this gives us a complete understanding of the irreducible representations of the groups $M_3$ and $M_4$, along with their global representation zeta functions.   
\end{abstract}
{\bf MSC Classification:} 20C15, 20F19\\
{\bf Keywords:} representation growth, nilpotent groups, representation zeta function, constructive method 

\section{Introduction}

Representation growth is an area within group theory where one studies (usually infinite) groups by considering the sequence of the number of (equivalence classes of) irreducible complex representations of degree $n$ for all $n \in \mathbb{N}.$ More formally, for a group $G$ and for all $n \in \mathbb{N}$ let $r_n(G)$ be the number of irreducible representations of degree $n$. If all $r_n(G)$ are finite, we can study this sequence by embedding them as coefficients in a zeta function.

Throughout this paper, all representations are complex representations. Two representations of a group $G,$ say $\rho_1$ and $\rho_2,$ are said to be \emph{twist-equivalent} if $\rho_1 = \chi\otimes \rho_2$ where $\chi$ is a representation of $G$ of degree 1. It is easy to see that there is only one twist isoclass of degree 1. For a finitely generated torsion-free nilpotent group $G$ all $r_n(G)$ are infinite. However, Lubotzky and Magid \cite{LM} show that if we redefine $r_n(G)$ to be the equivalence classes of representations of $G$ of degree $n$ up to twist equivalence then all $r_n(G)$ are finite. These are the objects we count in this paper. Additionally, they show that each irreducible representation of $G$ is twist equivalent to one that factors through a finite quotient of $G$.

Let $G$ be a $\T$-group and let the complex series
\begin{equation}
\zeta_{G}^{irr}(s)=\sum_{n=1}^{\infty} r_n(G)n^{-s}.
\end{equation}
We call $\zeta_{G}^{irr}(s)$ the \emph{(global) representation zeta function} of $G$. Additionally, let
\begin{equation}
\zeta_{G,p}^{irr}(s)=\sum_{n=0}^{\infty} r_{p^n}(G)p^{-ns}.
\end{equation}
We call this the \emph{$p$-local representation zeta function} of $G$.
Since all irreducible representations factor through a finite quotient and finite nilpotent groups are direct products of their Sylow $p$-subgoups, it follows that
\begin{equation}
\zeta_{G}^{irr}(s) = \prod_p\zeta_{G,p}^{irr}(s).
\end{equation}

In \cite[Theorem D]{Voll1}, Voll shows that for almost all primes,  $p$-local representation zeta functions of a $\T$-group $G$ satisfy a functional equation in the variable $p$, depending only on the Hirsch length of the derived subgroup of $G$. This functional equation has been refined in 
\cite[Theorem A]{SV} to include number theoretic information for groups that arise as, for a given number field, integer points of unipotent group schemes.

Define the group commutator $[a,b] = aba^{-1}b^{-1}$ and let $M_n := \langle a_1, \ldots, a_n,b \ | \ [a_i,b]=a_{i+1} \rangle$, where all commutators that do not appear in (or follow from) the group relations are trivial. This paper is a continuation of \cite{MaxClassI} in which we construct both the irreducible representations and the $p$-local zeta functions of the family of groups $M_n$ for almost all primes $p.$ While in the previous paper we calculated the $p$-local zeta functions when $p$ was non-exceptional, that is when $p \geq n$, in this paper we study some cases when $p$ is indeed exceptional.  For the rest of the paper, and with a slight abuse of notation, we denote the subgroup $\langle  a_{n-k+1}, a_{n-k+2}, \ldots, a_{n},b \rangle \leq M_n$  as $M_k.$ Our choice of group is notable as it is, in some sense and informally, the simplest family of $T$-groups of arbitrary nilpotency class and whose complexity is wholly derived from its nilpotency class, and not for any other number theoretic reasons. 

The study of representation growth, and asymptotic group theory in relation to the family of groups $M_n$,
has garnered significant interest within the mathematical community. See the prequel paper \cite[Introduction]{MaxClassI} for a comprehensive discussion of this, or such papers as \cite{SnockenThesis, zordan2022univariate, Rossmann16, Rossmann20, LMar,JZ,Voll3,Avnietal, Avni, Berman21} directly.

As to the knowledge of the author, there is no other ``small'' prime $p$-local representation zeta functions that appear in the literature. Indeed, both the Kirillov orbit method, first applied by Voll \cite{Voll1} to this use case, and Rossmann's topological representation zeta function machinery \cite{ROSSMANN2016210}, inherently exclude a finite number of primes and thus are unable to fully complete the calculation of global representation zeta functions in general. We stress that the aforementioned methods rely on deeper mathematical machinery and are applicable in a wide variety of cases. The method in this paper, which is much more elementary and deals with non-uniformity in the calculations in a largely ad-hoc manner, is applicable to small exceptional primes due to the lack of assumptions and mathematical sophistication. 


\subsection{Main Result}

Let $p_*$ be a fixed prime. The main result of this paper is the calculation of the $p_*$-local representation zeta function for $M_{p_*+1}$, which notably includes the $2$-local representation zeta function of $M_3$. We also calculate the $2$-local zeta function for $M_4$. Combining these results with the non-exceptional $p$-local representation zeta functions in \cite{MaxClassI}, this gives us the global representation zeta functions for $M_3$ and $M_4$. Note the uniformity in all of the $p$-local zeta functions for both groups. It is striking that the exceptional prime-local zeta functions are of the same form as the non-exceptional prime-local zeta functions. We have the following theorem:

\begin{theorem}
Let $p$ be a prime. Then the $p$-local representation zeta function of the group $M_{p+1}$ is
\[ \zeta^{irr}_{M_{p+1},p}(s) =\frac{(1-p^{-s})^2}{(1-p^{((p+1)-2)-s})(1-p^{1-s})}.
\]
\end{theorem}

\begin{theorem}
The $2$-local representation zeta function of the group $M_4$ is 

\[ \zeta^{irr}_{M_{4},2}(s) = \frac{(1-2^{-s})^2}{(1-2^{1-s})(1-2^{2-s})}.
\]
\end{theorem}

Combining the previous two theorems with the main result from \cite{MaxClassI} and \cite[Theorem 5]{NM} we have the following corollary.

\begin{corollary}
Let $n \in \{2,3,4\}.$ Then the global representation zeta function of $M_n$ is
\[ \zeta^{irr}_{M_n}(s)= \dfrac{\zeta(s-1)\zeta(s-(n-2))}{(\zeta(s))^2}.
\]
\end{corollary}

\section{Important Results from Prequel}

In order to keep this paper relatively self-contained, we give a list of results and definitions that appear in \cite{MaxClassI} that are used in this paper. For the proofs of the following results, see the aforementioned paper.

\begin{definition}\label{def:SpN} 
Let $\mu_p^{\infty}$ be the set of all complex $p^\ell$th roots of unity for all $\ell \in \mathbb{N}$ and $\mu_p^k$ be the set of all $p^k$th roots of unity (and note that $\mu_p^k\backslash \mu_p^{k-1}$ is the set of primitive $p^k$th roots of unity). Define $s:\mu^\infty_p \to \mathbb{N}$ by $s(\lambda)=k$ if and only if $\lambda \in \mu_p^k \backslash \mu_p^{k-1}.$
\end{definition}

Let $T_0(0)=1, \  T_0(j)=1,$ and $T_j(0)=0$ for $j \in \mathbb{N}$ and recursively define $T_{k}(j)=\sum_{l=1}^j T_{k-1}(l)=T_{k}(j-1)+T_{k-1}(j)$  for $k \in \mathbb{N}.$ The next lemma lists some properties of these numbers that are needed for this paper. 

\begin{lemma}\cite[Lemma 2.4]{MaxClassI} Let $i,j,k,b \in \mathbb{N}$ and $T_{k}(j)$ be defined as above.
\begin{enumerate}[i.] \label{tlemma}
\item $T_{k}(i)= {{i+k-1}\choose{k}}=\frac{i(i+1)\ldots(i+k-1)}{k!}.$
\item Let $p > k.$ Then for any $b \in \mathbb{N}$ and $\alpha$ such that $1 \leq \alpha \leq p-1$ we have $T_k(\alpha p^b + j) = T_k(j) \mod p^b.$
\item If $p>k$ then $T_k(p^N-1)=0 \mod p^N.$
\end{enumerate}
\end{lemma}

As a corollary of (ii) we have the following.

\begin{corollary}\cite[Corollary 2.5]{MaxClassI}\label{termequality}
Let $p$ be a prime, let $k <p$, let $N \geq 1,$ let $1 \leq m \leq N,$ and let $\alpha \in \mathbb{N}$ such that $p \nmid \alpha.$ 
    
Then, for all $\beta$ such that $1 \leq \beta < p^m$ and all $j$ such that $0 \leq j \leq p^{N-m}-1,$ we have that $\alpha p^m T_k(\beta p^{N-m}+j) =  \alpha p^m T_k(j) \mod p^N.$ 

\end{corollary}

\begin{lemma}\cite[Lemma 3.1]{MaxClassI}\label{formx}
For $ 1 \leq i \leq n-1$  we have that $\lambda_{i,j} =\prod_{k=i}^{n} \lambda_{k}^{T_{k-i}(j-1)}$ and thus the matrix $x_i$ has the structure
\begin{equation}x_{i} =\left( \begin{array} {cccc} \lambda_{i} &  &  &   \\  &  \prod_{k=i}^{n} \lambda_{k}^{T_{k-i}(1)} & &  \\  &
 & \ddots &  \\  &  &
 &  \prod_{k=i}^{n} \lambda_{k}^{T_{k-i}(p^N-1)} \end{array} \right).\end{equation}
Moreover we have that
\begin{equation}\label{cor1}
\lambda_i^{p^N} \prod_{k=i+1}^n \lambda_k^{T_{k-i}(p^N-1)} = 1.
\end{equation}
\end{lemma}

\begin{definition}
The matrices $x_1, \ldots, x_n, y$ are in \emph{standard form} if the $x_i$ are in the form of Lemma \ref{formx} and $y$ is in the form
\begin{equation}\label{formy}
y =\left( \begin{array} {cccc} 0 &  &  & 1  \\ 1 & \ddots & &  \\  & \ddots
 & \ddots &  \\  &  & 1
 & 0 \end{array} \right).
\end{equation} We say $\rho$ is in \emph{standard form} if, under a chosen basis, the matrices $x_1, \ldots, x_n, y$ are in standard form. We remind the reader that in standard form we set $\lambda_1=1$ by twisting.
\end{definition}

We set up notation for entries of our matrices. Let $\lambda_{i,j}$ be the $j$th diagonal entry of $x_i$ and let $\lambda_i=\lambda_{i,1}$. Let $V_{p^k}$ be the subspace spanned by $\langle y \rangle \cdot (e_1+e_{p^{k}+1}+ \ldots +e_{(p^{N-k}-1)p^{k}+1})$. Also, we define the $n$-tuples $\Lambda_n(k):=(\lambda_{1,k}, \ldots, \lambda_{n,k})$ and $\Lambda^\prime_n(k):=(\lambda_{2,k}, \ldots, \lambda_{n,k})$ where $k$ is considered mod $p^N.$ 

Note that we view $\Lambda^\prime$ as $\Lambda$ when are representation is restricted to $M_{n-1}$ and thus the following lemma also holds for $\Lambda^\prime(k)$.

\begin{lemma}\cite[Lemma 4.1]{MaxClassI}\label{tupleequality}
For any $k_1,k_2$ if $\Lambda_n(k_1)=\Lambda_n(k_2)$ then $\Lambda_n(k_1+1)=\Lambda_n(k_2+1).$
\end{lemma}

\begin{corollary}\cite[Corollary 4.2]{MaxClassI}\label{minimalV}
Let $j$ be the minimal power such that $\Lambda_n(k)=\Lambda_n(\beta p^j +k)$ for all $\beta$ such that $0 \leq \beta \leq p^{N-j}-1$ and for all (and equivalently by Lemma \ref{tupleequality}, for one) $k.$ Then $V_{p^j}$ is a stable subspace of $\rho$ and $V_{p^{j-1}}$ is not stable.
\end{corollary}

We define notation to this effect. Let $H \leq M_n$ and let $\mathcal{V}(\rho|_H)$ be the the minimal stable subspace $V_{p^j},$ as in Corollary \ref{minimalV}, of $\rho|_H.$ We say that $\mathcal{V}(\rho)=\mathcal{V}\big(\rho(M_n)\big).$

\begin{corollary}\cite[Corollary 4.3]{MaxClassI}\label{minimalpower}
 The following are equivalent:
 \begin{enumerate}
\item The number $j$ is minimal such that $\Lambda_n(1)=\Lambda_n(p^j+1)$
\item $\V (\rho)=V_{p^j}.$
\end{enumerate} 
  
\end{corollary}

\begin{corollary}\cite[Corollary 4.4]{MaxClassI}\label{inclusionV}
Let $\rho:M_n \to GL_{p^N}(\mathbb{C})$ be a representation. If $\V(\rho|_{M_k})=V_{p^{j}}$ for  all $k < n$ then there is some $\ell \geq j$ such that $\V(\rho)=V_{p^\ell}.$
\end{corollary}

We know that if $V_{p^k}$ is $\rho$-stable then so is $V_{p^j}$ for $j \geq k.$ Thus, we obtain the following corollary:
\begin{corollary}\cite[Corollary 4.5]{MaxClassI}\label{VpN-1}
Let $\rho$ be a representation of $M_n.$ The representation $\rho$ is irreducible if and only if $V_{p^{N-1}}$ is not $\rho$-stable.
\end{corollary}
Throughout this paper we use Corollary \ref{VpN-1} to check if a representation $\rho$ is irreducible. We use Corollary \ref{minimalpower} to determine the number of isomorphic representations in standard form in one twist isoclass. 

\begin{definition}
Let $\rho$ be irreducible and let $x_i,y,$  for $i$ such that $1 \leq i \leq n,$ be in standard form as defined earlier in the section. A \emph{standardization} is a matrix $P$ such that, up to twisting, $PyP^{-1}$ and $Px_iP^{-1}$ for $i=1 \ldots n$ are in standard form. The representations $\rho$ and $P\rho P^{-1}$ (note that $Px_1P^{-1}$ may not be in standard form) are said to be equivalent under \emph{standardization}.
\end{definition}

\begin{lemma}\cite[Lemma 5.4]{MaxClassI}\label{shout}
Let $S_\rho$ be the twist isoclass represented by $\rho$ and let $\V(\rho|_{M_{n-1}})=V_{p^{m}}.$ Then there are $p^m$ representations in standard form in $S_\rho$ that are standardization equivalent to $\rho.$
\end{lemma}

We use the following proposition and lemma to help count standardization equivalent representations in each isoclass.

\begin{proposition}\cite[Proposition 6.1]{MaxClassI}\label{form}
Let $p\geq n$ and $\rho$ be a $p^N$-dimensional representation of $M_n$ with corresponding matrices in standard form. Then $\rho$ is irreducible if and only if there exists a $\lambda_i$  such that $s(\lambda_i)=N,$ where $2 \leq i \leq n.$
\end{proposition}

We slightly change the form of this lemma from its version in \cite{MaxClassI}.

\begin{lemma}\cite[Lemma 7.1]{MaxClassI}\label{nonexshout}
For $p\geq n-1$ let $\rho$ be an irreducible $p^N$-dimensional representation of $M_n$ and let $k=\max\{s(\lambda_3),\ldots s(\lambda_n)\}.$  Then there are $p^k$ representations in standard form equivalent to $\rho$ under twisting and standardization.
\end{lemma}

\section{The $p$-local Representation Zeta Function for $M_{p+1}$}

For a prime $p$, we study the $p$-local representations of $M_{n}$ when $n=p+1$. We calculate the exceptional prime representation zeta function $\zeta^{irr}_{M_{p+1},p}(s).$ Note that, unlike the non-exceptional calculation in \cite{MaxClassI},  $p$ is fixed by our choice of group for this calculation.

 Let $\rho$ be a $p^N$-dimensional representation. We will determine the choices of $\lambda_i$ for which $\rho$ is irreducible. We can choose a basis of the form in Lemma \ref{formx}.

For this section let $p > 2$. The calculation when $p=2$ for the group $M_3$ is similar to, and more straightforward than, the calculation below and we leave the details to the reader. 

We now divide this calculation into two cases: when $s(\lambda_{p+1})=N$ and when $s(\lambda_{p+1})\leq N-1$. Furthermore, we break the second case into two sub-cases: when there is a $\lambda_i$ with $3\leq i \leq p$ such that $s(\lambda_i)=N$ and when there is no such $\lambda_i.$ 

Note that, since $p$ is not exceptional when considering $\rho|_{M_p}$ (and we remind the reader that $M_p=\langle x_2, \ldots, x_p, y \rangle$), we can apply Lemma \ref{nonexshout} when determining the number of representations standardization equivalent to some irreducible $\rho.$

\textbf{Case 1} Assume that $s(\lambda_{p+1})=N.$\\
By \cite{NM}, we have that $\rho|_{M_2}$ is an irreducible representation and thus $\rho$ is irreducible.  We now must determine for which choices of $\lambda_i$ the representation is well defined.

By Lemma \ref{tlemma}(iii), it is clear
that we can use Equation \ref{cor1} in Lemma \ref{formx} to show  that $s(\lambda_i) \leq N $ for $i \neq 2$ and
\begin{equation}
\lambda_2^{p^N}\prod_{k=3}^{p+1}\lambda_k^{T_{k-1}(p^N-1)}=1.
\end{equation}

Since $s(\lambda_i) \leq N $ for $3 \leq i \leq p$ by Lemma \ref{tlemma}(iii) the preceding equation simplifies to
\begin{equation} \label{welldefinedp}
\lambda_2^{p^N}\lambda_{p+1}^{T_p(p^N-1)} = 1.
\end{equation}
\noindent Cancelling out the $p$ from the denominator of $T_p(p^N-1)$, we have that
\begin{equation}\label{Tp}
T_p(p^N-1) = \alpha p^{N-1}
\end{equation} for some $\alpha$ coprime to $p$. Thus $s(\lambda_{p+1}^{\alpha p^{N-1}})=1$ and by Equation \ref{welldefinedp} we have that $s(\lambda_2)=N+1$. There are $p^N$ choices for $\lambda_2$ so that Equation \ref{welldefinedp} holds. Thus, there are $(1-p^{-1})p^N$ choices for $\lambda_{p+1}$ and $p^N$ choices for each $\lambda_{i}$ where $2 \leq i \leq p.$ By Lemma \ref{nonexshout} we must divide by $p^N$ to take standardization into account. Therefore in this case there are
\begin{align}\label{pcase1}
(1-p^{-1})p^N p^{(p-1)N}p^{-N} =(1-p^{-1})p^{(p-1)N}
\end{align}
twist isoclasses. Note that the right hand side of Equation $\ref{pcase1}$ is also the contribution to $r_{p^N}$ in the non-exceptional case in \cite{MaxClassI} for when $s(\lambda_{p+1})=N.$\\

\textbf{Case 2} Now assume $s(\lambda_{p+1}) \leq N-1.$\\
It is clear, since $T_i(p^N-1) = 0 \mod p^N$ for $i < p$ by Lemma \ref{tlemma}(iii) and since $\lambda_{p+1}^{T_{p}(p^N-1)} = 1$ by Equation \ref{Tp}, that we can say that $s(\lambda_i) \leq N $ for $2 \leq i \leq p+1.$ We now break this case into subcases.

\textbf{Case 2.1} For $i$ such that $3 \leq i \leq p,$ assume one of $s(\lambda_i)=N,$ say $\lambda_k.$ Then, since $p \geq k,$ by Proposition \ref{form} we have that $\rho|_{M_{p+1-k+2}}$ is an irreducible representation and thus $\rho$ is irreducible. In this case there are $(1-p^{-(p-2)})p^{(p-2)N}$ choices for $\lambda_i,$ $p^N$ choices for $\lambda_2,$ and $p^{N-1}$ choices for $\lambda_{p+1}.$ By Lemma \ref{nonexshout} we must divide by $p^N$ to take standardization into account. Thus there are
\begin{align}
(1-p^{-(p-2)})p^{(p-2)N}p^{N-1}p^Np^{-N} = (1-p^{-(p-2)})p^{(p-1)N-1}
\end{align}
twist isoclasses in this case. We note that the contribution to $r_{p^N}$ in this case is the same contribution to $r_{p^N}$ for non-exceptional primes \cite[Section 7]{MaxClassI}.\\

\textbf{Case 2.2} Assume $s(\lambda_i) \leq N-1$ where $3 \leq i \leq p.$\\
Note that in this case $\rho|_{M_p}$ has $V_{p^{N-1}}$ as a proper stable subspace so by Lemma \ref{VpN-1} it is not irreducible. If $s(\lambda_{p+1})=0$ then $M_{p+1}$ is isomorphic to $M_p$ and by Proposition \ref{form} the representation $\rho$ is irreducible if and only if $s(\lambda_2)=N$.

Now let $s(\lambda_{p+1})\geq 1.$ We choose $\lambda_*$ such that $s(\lambda_*)=N$ and write each $\lambda_i$ in terms of it; that is, let $\lambda_i = \lambda_*^{\alpha_i p^{m_i}},$ $p \nmid \alpha_i$, $m_{2} \geq 0$, and $m_i \geq 1$ for $3 \leq i \leq p+1.$


We appeal to Lemma \ref{minimalpower} and determine when $\langle y, x_1\rangle$ does not have  $V_{p^{N-1}}$ as a proper stable subspace.   This is the case exactly when $\lambda_{1,1} \neq \lambda_{1,p^{N-1}+1}.$ 

Consider $\lambda_{1,p^{N-1}+1}$. If $N=1,$ then in order for $\rho$ not to be trivial we have that $s(\lambda_2)=1$ and it is easily verified that $x_1$ is not scalar and thus $\rho$ is irreducible. Now, for $N \geq 2,$ we have that
\begin{align}
\Lambda :={}& \text{log}_{\lambda_*}(\lambda_{1,p^{N-1}+1})\\
={}& \alpha_2 p^{m_2} (p^{N-1}) + \alpha_3p^{m_3}\frac{(p^{N-1})( p^{N-1}+1)}{2}+\ldots \nonumber\\
&+ \alpha_{p+1} p^{m_{p+1}} \frac{(p^{N-1}) \ldots ( p^{N-1}+p-1)}{p!} \mod p^N.\nonumber
\end{align}

By Corollary \ref{termequality} and keeping in mind that $m_{i} \geq 1$ for $3 \leq i \leq p+1$ this simplifies to the following:

\begin{align}
\Lambda ={}& \alpha_2 p^{m_2} p^{N-1} + \alpha_{p+1} p^{m_{p+1}-1} \frac{p^{N-1} \ldots (p^{N-1}+p-1)}{(p-1)!} \mod p^N \\
={}&p^{N-1} \Big( \alpha_2 p^{m_2}+\alpha_{p+1} p^{m_{p+1}-1}\frac{(p^{N-1}+1) \ldots (p^{N-1}+p-1)}{(p-1)!}\Big) \mod p^N \nonumber
 \end{align}
Note that the last term has a denominator of $(p-1)!$ and the exponent of $p$ in that term is now $m_{p+1}-1.$ 

We want $\rho$ to be irreducible. Thus, by Wilson's Theorem it must be that
\begin{equation}\label{pqpqp}
\alpha_2 p^{m_2} + \alpha_{p+1}p^{m_{p+1}-1} \neq 0 \mod p.
\end{equation}
We now enumerate the cases when we do not have a factor of $p,$ and thus an irreducible representation. By Equation \ref{pqpqp} this is precisely when $m_{p+1}=1$ or $m_2=0$ except when $m_{p+1}= 1, m_2 = 0,$ and $\alpha_2 \neq -\alpha_{p+1} \mod p.$

We still need to take standardization into account. Therefore, by Lemma \ref{nonexshout}, we must divide our count, if we enumerated the representations in this case at this stage, by $p^{m_*}$ where $m_*= \text{max}\{s(\lambda_3), \ldots, s(\lambda_{p+1})\}$. Note that, since $p$ is non-exceptional when considering $\rho|_{M_{p}},$ the standardization behaviour is the same as in the non-exceptional case. We will count these later.\\

This ends the case distinctions.\\

We note that the only difference between the $r_{p^N}$ for this exceptional prime  and the $r_{p^N}$ for non-exceptional primes is the situation when we can choose $\lambda_{2}$ and $\lambda_{p+1}$ such that (still thinking of all $\lambda_i$ written as powers of $\lambda_*$) $m_{p+1}=1$ and $m_2 \geq 1$, which gives us additional irreducible representations, and when $m_{p+1}= 1,  m_2 = 0,$ and $\alpha_2 \neq -\alpha_{p+1} \mod p,$ which gives us representations that are no longer irreducible. Therefore, starting with $r_{p^N}$ calculated for non-exceptional primes, we can add the cases where our choices of $\lambda_i$ give us additional representations and subtract the cases where we lose representations.

Let $C$ be $r_{p^N}$ for non-exceptional primes, that is the sum in \cite[Section 7]{MaxClassI} The situation where $m_{p+1}=1$ and $m_{2} \geq 1$ does not correspond to irreducible representations for non-exceptional primes, but does for exceptional primes.  There are $(1-p^{-1})p^{N-1}$ choices for $\lambda_{p+1}$ and $p^{N-1}$ choices for $\lambda_{2}$ in this case. Remembering that we assumed that $s(\lambda_i) \leq N-1$ for $3 \leq i \leq p$ then there are $p^{(p-2)(N-1)}$ choices for these $\lambda_i.$ By Lemma \ref{nonexshout} we must divide by $p^{N-1}$ to take standardization into account. Therefore we must add

\begin{equation}\label{h}
(1-p^{-1})p^{N-1} p^{N-1} p^{(p-2)(N-1)}p^{-(N-1)} = (1-p^{-1})p^{(p-1)(N-1)}
\end{equation} to $C$.

The situation where $m_{p+1}=1,$ $m_2 = 0,$ and $\alpha_2 = -\alpha_{p+1} \mod p$  does correspond to irreducible representations for non-exceptional primes, but does not for exceptional primes. There are $(1-p^{-1})p^{N}$ choices for $\lambda_2$ and, given our choice for $\lambda_2$, there are $p^{N-2}$ choices for $\lambda_{p+1}$ in this case. Remembering that we assumed that $s(\lambda_i) \leq N-1$ for $3 \leq i \leq p$ there are $p^{(p-2)(N-1)}$ choices for these $\lambda_i.$ By Lemma \ref{nonexshout} we must divide by $p^{N-1}$ to take standardization into account. Therefore we must subtract
\begin{equation}\label{hh}
(1-p^{-1})p^{N} p^{N-2}p^{(p-2)(N-1)}p^{-(N-1)} =(1-p^{-1})p^{(p-1)(N-1)}
\end{equation} from $C$. Notice that $(\ref{h})=(\ref{hh}).$ Therefore
\begin{equation}
r_{p^N} = C
\end{equation} and
\begin{equation}
\zeta^{irr}_{M_{p+1},p}(s) =\frac{(1-p^{-s})^2}{(1-p^{((p+1)-2)-s})(1-p^{1-s})}
\end{equation} by \cite[Equation 31]{MaxClassI}.\\

This result, and the result from \cite{MaxClassI}, gives us the entire irreducible representation theory, as well as the representation zeta function, of $M_3.$ In fact, we can say that

\begin{equation}\label{M32}
\zeta^{irr}_{M_3}(s)= \left(\frac{\zeta(s-1)}{\zeta(s)}\right)^2.
\end{equation}

\section{The $2$-local Representation Zeta Function for $M_4$}\label{sub:Mp+1}

We now have a complete understanding of the irreducible representations of $M_3.$ The aim of this section is to do the same for $M_4.$ Our previous work leaves us with only one $p$-local zeta function to calculate; the previous section calculates the $3$-local zeta function and $2$ and $3$ are the only exceptional primes. Therefore once we calculate the $2$-local representation zeta function we have $\zeta^{irr}_{M_4}(s)$ in its entirety. 

The complexity of this calculation lies partly in the inability to use Lemma \ref{nonexshout} in many cases. Thus, some work is to be done to calculate the correct factor by which we are overcounting. 

For ease of computation, we calculate $r_2(M_4)$ separately later in this section. Until noted otherwise we assume the condition that $N \geq 2$. In keeping with the style of the general cases earlier, and for elucidation if one wishes to generalize this calculation, we do not simplify the expressions $(1-2^{-1})$ to $2^{-1}$ as far in the calculation as possible.

Let $\rho:M_4 \to GL_{2^N}(\mathbb{C})$ be a representation. By Equation \ref{cor1} in Lemma \ref{formx} we have that
\begin{equation}
\lambda_4^{2^N}=1, \label{c}
\end{equation}
\begin{equation}
\lambda_3^{2^N}\lambda_4^{T_2(2^N-1)}=1, \label{cc}
\end{equation} and
\begin{equation}
\lambda_2^{2^N}\lambda_3^{T_2(2^N-1)}\lambda_4^{T_3(2^N-1)}=1.\label{ccc}
 \end{equation} Therefore, by Equation \ref{c}, we have that $s(\lambda_4) \leq N.$

Before we begin counting twist isoclasses we must determine the possible depths of $\lambda_2$ and  $\lambda_3$. We remind the reader of Equation \ref{Tp}. Assume $ s(\lambda_4) \leq N-1.$ Then   $\lambda_4^{T_2(2^N-1)}= \lambda_4^{T_3(2^N-1)} = 1$ and by Equation \ref{cc} we have that $\lambda_3^{2^N}=1$ and thus $s(\lambda_3) \leq N.$ If $s(\lambda_3) \leq N-1$ then $\lambda_3^{T_2(2^N-1)}=1$ and by Equation \ref{ccc} we have that $\lambda_2^{2^N}=1$ so $s(\lambda_2) \leq N.$ If $s(\lambda_3)=N$ then $\lambda_3^{T_2(2^N-1)}= \lambda_3^{-2^{N-1}}$ and thus $s(\lambda_3^{-2^{N-1}})= 1$. By Equation \ref{ccc} we have that $\lambda_2^{2^N} = \lambda_3^{2^{N-1}}$ and thus $\lambda_2^{2^N}$ must satisfy this equation. So $\lambda_2^{2^N}=-1$ and $s(\lambda_2)=N+1.$

Now assume $s(\lambda_4) = N.$ Then $\lambda_4^{T_2(2^N-1)}= \lambda_4^{T_3(2^N-1)}= \lambda_4^{-2^{N-1}}$ and thus $s(\lambda_4^{-2^{N-1}})= 1$. By Equation \ref{cc} we have that
\begin{equation} \label{pnun}
\lambda_3^{2^N}=\lambda_4^{2^{N-1}}
\end{equation}
and thus $\lambda_3^{2^N}$ must satisfy this equation. So $\lambda_3^{2^N}=-1$ and $s(\lambda_3)=N+1.$ We have that $\lambda_3^{T_2(2^N-1)}= \lambda_3^{-2^{N-1}}$ and by Equations \ref{ccc} and \ref{pnun}  we have that $\lambda_2^{2^N}= \lambda_3^{2^{N-1}}\lambda_4^{2^{N-1}} = \lambda_3^{2^N+2^{N-1}}= \lambda_3^{(1+2)2^{N-1}}.$ Note that we leave $(1+2)$ in this form since we wish to keep the form $(1+p).$ Thus $s(\lambda_3^{(1+2)2^{N-1}})=2$ and $\lambda_2^{2^N}$ must satisfy Equation \ref{ccc}. So $\lambda_2^{2^N}= \pm \sqrt{-1}$ and  $s(\lambda_2)=N+2.$

\begin{table}[htpb]\caption{Table of Cases for $M_4$}\label{tab:M4cases}
\scalebox{0.85}{
\begin{tabular}  {|c||c|c|c|p{3cm}|p{3.5cm}|}
 \hline
Case & $s(\lambda_4)$ & $s(\lambda_3)$ & $s(\lambda_2)$ & Other Conditions & No. of twist isoclasses, $N \geq 2$\\
\hline
1 & $=N$ & $=N+1$ & $=N+2$ & & $(1-2^{-1})^42^{2N+3}$\\ \hline
2 & $=N-1$ & $=N$ & $=N+1$ & $\alpha_3 = 3 \mod 4$ & $(1-2^{-1})^3 2^{2N}$\\ \hline
3 & $=N-1$ & $\leq N-1$ & $\leq N$ & & $(1-2^{-1})2^{2N-2}$\\ \hline
4 & $\leq N-2$ & $=N$ & $=N+1$ & & $(1-2^{-1})^2 2^{2N-1}$\\ \hline
5 & $\leq N-2$ & $=N-1$ & $=N$ & & $0$\\ \hline
6 & $\leq N-2$ & $\leq N-2$ & $=N$ & & See Table \ref{tab:Case6} on page \pageref{tab:Case6}\\ \hline
7 & $\leq N-2$ & $=N-1$ & $\leq N-1$ &  & See Table \ref{tab:Case7} on page \pageref{tab:Case7}\\ \hline
8 & $\leq N-2$ & $\leq N-2$ & $\leq N-1$ & & $0$\\
\hline
\end{tabular}}

\end{table}

\begin{table}[htpb]\caption{Case 6 of Table \ref{tab:M4cases}}\label{tab:Case6}
\scalebox{0.85}{
\begin{tabular}{|c||c|p{7cm}|p{5cm}|}
\hline
$N$ & Case & Relationship of $s(\lambda_3)$ and $s(\lambda_4)$ & No. of twist isoclasses \\
\hline
$=2$ & 6.4 & $s(\lambda_3)=s(\lambda_4)=0$ & 2\\
\hline
\multirow{2}{*} {=3} & 6.2 & $s(\lambda_3)\leq 1, s(\lambda_4)=1$ &  $(1-2^{-1})^2 2^3$\\
 & 6.4 & $s(\lambda_4)=0$ & $(1-2^{-1})2^3(1+(1-2^{-1}))$ \\


\hline
\multirow{5}{*} {$\geq 4$} & 6.1 & ${s(\lambda_3)>s(\lambda_4)+1}, {s(\lambda_3) \geq 2}, \linebreak[2] {s(\lambda_4) \neq 0}$ & $\big[(1-2^{-1})2^N\big((2^{N-4}-1){\qquad} \linebreak[4] -(1-2^{-1})(N-4)\big)\big]$\\
& 6.2 & $s(\lambda_3) < s(\lambda_4)+1, s(\lambda_4)\neq 0$ & $(1-2^{-1})2^N(2^{N-3}-2^{-1})$\\
& 6.3 &$s(\lambda_3)=s(\lambda_4)+1, s(\lambda_4) \neq 0$ & $(1-2^{-1})^2 2^N (2^{N-2}-2)$\\
& 6.4 &$s(\lambda_4)=0$ & $(1-2^{-1})^2 2^N(1+(1-2^{-1})(N-2))$\\

\hline
\end{tabular}
}
\end{table}

\begin{table}[htpb]\caption{Case 7 of Table \ref{tab:M4cases}}\label{tab:Case7}
\scalebox{0.85}{
\hspace{4cm}
\begin{tabular}{|c||c|p{5cm}|}
\hline
$N$ & Case &  No. of twist isoclasses\\
\hline
$=2$ & &1\\
\hline
\multirow{2}{*} {$\geq 3$} & 7.1 & $(1-2^{-1})2^{2N-4}$\\
 & 7.2 & $(1-2^{-1})^2 2^{2N-2}$\\
\hline
\end{tabular}
}
\end{table}

We break our computation into eight cases, with Cases 6 and 7 being further broken down into subcases.  Tables 1,2, and 3 show, repsectively, the number of twist isoclasses in each case, for the subcases of Case 6, and for the subcases of Case 7. We leave the computation of Cases 1,3, and 4 to the reader; these follow almost immediately from previous computations. 

\vspace{1cm}

\textbf{Case 2}: By Case $2.2$ of Section \ref{sub:Mp+1} we have that $\rho|_{M_3}$ is reducible. Appealing to Lemma \ref{VpN-1} we must check whether $V_{2^{N-1}}$ is a stable subspace of $\langle y,x_1 \rangle.$ We write each root of unity in terms of a primitive $2^{N+1}$th one. Let $\lambda_* = \lambda_2, \  \lambda_i = \lambda_*^{\alpha_i 2^{m_i}}$ for some $\alpha_i$ such that $2 \nmid \alpha_i, m_4=2,$ $m_3=1,$ and $i \in \{3,4\}$.

Using Corollary Corollary \ref{minimalpower} and noting that $2\cdot(2^{N-1})^2 = 0 \mod 2^{N+1}$ for $N=2$, consider $\lambda_{1,2^{N-1}+1}$:

\begin{align}
& \log_{\lambda_*}(\lambda_{1,2^{N-1}+1})\\
={}& (2^{N-1}) + 2^{1-1}\alpha_3 (2^{N-1})(2^{N-1}+1)\nonumber\\
&+ \alpha_4 2^{2-1} \frac{(2^{N-1})(2^{N-1}+1)(2^{N-1}+2)}{3} \mod 2^{N+1}\nonumber\\
={}& 2^{N-1}\left(1+\alpha_3+\alpha_4 2^1 \frac{2}{3}\right) \mod 2^{N+1}\nonumber\\
={}& \text{log}_{\lambda_*}(\lambda_{1})+ 2^{N-1}\left[1+\alpha_3\right] \mod 2^{N+1}\nonumber
\end{align}

So the expression in the square brackets above is a multiple of 4 if and only if $V_{2^{N-1}}$ is a $\langle y,x_1 \rangle $-stable subspace.  Let $Q$ be the aforementioned expression. It is clear that $Q=0 \mod 4$ precisely when $\alpha_3=3 \mod 4.$ This means that we are only free to choose half of the elements of $S_2^N/S_2^{N-1}$ for $\lambda_3.$ Thus, there are $(1-2^{-1})2^{N-1}$ choices for $\lambda_3$, $(1-2^{-1})2^{N+1}$ choices for $\lambda_2$, and $(1-2^{-1})2^{N-1}$ choices for $\lambda_4.$ Since $\rho|_{M_3}$ is not irreducible it has at least $V_{2^{N-1}}$ as a stable subspace. But since $s(\lambda_4)=N-1,$ by Corollary \ref{inclusionV} we have that $\V(\rho|_{M_3})=V_{2^{N-1}}.$ Thus, by Lemma \ref{shout} we must divide by $2^{N-1}$ to take standardization into account. So in this case we have
\begin{align}\label{m4case2}
&(1-2^{-1})2^{N-1}(1-2^{-1})2^{N+1}(1-2^{-1})2^{N-1} 2^{-(N-1)}\\
={}&(1-2^{-1})^3 2^{2N}\nonumber
\end{align}
twist isoclasses.

\textbf{Cases 5 and 6}: We note  $s(\lambda_2)=N$ and $s(\lambda_4)\leq N-2$ for both cases. We have, by Case 2.2 of Section \ref{sub:Mp+1}, that $\rho|_{M_3}$ has $V_{2^{N-1}}$ as a proper stable subspace. Appealing to Lemma \ref{VpN-1}, we check whether $V_{2^{N-1}}$ is a stable subspace of $\langle y,x_1 \rangle.$ We let $\lambda_*=\lambda_2$ and write each $\lambda_i$ as a power of $\lambda_*;$ that is, let $\lambda_4 = \lambda_*^{\alpha_4 2^{m_4}}$ and $\lambda_3 = \lambda_*^{\alpha_3 2^{m_3}}$ such that $2 \nmid \alpha_i$ $m_3 \geq 1,$  $m_4 \geq 2,$ and $i \in \{3,4\}.$
If $m_4=N$ then by Case 2.2 of Section \ref{sub:Mp+1} we have that $\rho$ is irreducible if and only if $m_3 \neq 1.$ If $m_3=N$ it is easy to show that $\log_{\lambda_*}(\lambda_{1,2^{N-1}+1}) \neq 1.$ We leave this to the reader. Assume that $m_3,m_4 \neq N.$

Appealing to Corollary \ref{minimalpower}, consider $\lambda_{1,2^{N-1}+1}$, noting that $2^{2N-2} = 0 \mod 2^N:$
\begin{align}
\Lambda :={}& \text{log}_{\lambda_*} (\lambda_{1,2^{N-1}+1})\\
 ={}& (2^{N-1}) + \alpha_3 2^{m_3-1} (2^{N-1})(2^{N-1}+1)\nonumber\\
 &+ \alpha_4 2^{m_4-1} \frac{2^{N-1}(2^{N-1}+1)(2^{N-1}+2)}{3} \mod 2^N\nonumber\\
  ={}& \text{log}_{\lambda_*}(\lambda_{1}) + 2^{N-1}\left[1+\alpha_3 2^{m_3-1}\right] \mod 2^N. \nonumber
\end{align}

So when the term in the square brackets above, say $Q$, is  not $0 \mod 2$ then $\lambda_{1} \neq \lambda_{1, 2^{N-1}+1}.$ It follows that $V_{2^{N-1}}$ is not a stable subspace of $\rho$ and therefore $\rho$ is irreducible. Thus $Q$ is $0 \mod 2$ when $m_3=1;$  that is when $s(\lambda_3)= N-1.$ So in Case 5 there are no irreducible representations.

If $m_3 \geq 2$ it is clear that $Q \neq 0 \modd 2.$ Thus, in Case 6 there are $2^{N-2}$ choices for $\lambda_4$, $2^{N-2}$ choices for $\lambda_3$, and $(1-2^{-1})2^N$ choices for $\lambda_2.$

We now need to analyze the standardization behaviour for this case. It is clear, since $\V(\rho|_{M_2})=V_{2^{s(\lambda_4)}}$ by Lemma \ref{nonexshout} and Corollary \ref{inclusionV}, that there are at least $2^{s(\lambda_4)}=2^{N-m_4}$ representations standardization equivalent to $\rho$. We now determine $\V(\rho|_{M_3})$ for each possible choice of $m_3$ and $m_4.$  Let $m_4 \neq N.$ We deal with the case $m_4=N$ in the next lemma. Also, note that we use the power of Corollary \ref{minimalpower} for this computation.

Consider, for some $k$ such that $1\leq k \leq m_4,$
\begin{align}\label{Case6}
&\text{log}_{\lambda_*}(\lambda_{2, 2^{N-k}+1})-\text{log}_{\lambda_*}(\lambda_{2,1})\\
        ={}&\alpha_3 2^{m_3} 2^{N-k}+\alpha_4 2^{m_4-1} 2^{N-k}( 2^{N-k}+1) \mod 2^N \nonumber \\
        ={}&2^{N-k}\left[\alpha_3 2^{m_3}+\alpha_4 2^{m_4-1}(2^{N-k}+1)\right] \mod 2^N. \nonumber
\end{align}

For the following lemma let $Q$ be the sum in the square brackets above. By Lemma \ref{minimalpower}, if $Q = 0 \mod 2^k$ then $\lambda_{2,1}=\lambda_{2, 2^{N-k}+1}$ and $V_{2^k}$ is a proper stable subspace of $\rho|_{M_3}.$

\begin{lemma}\label{shoutM4}
Let $m_4 \neq N$ and let $m_*=\min\{m_3,m_4-1\}.$ If $m_3 \neq m_4-1$ then $\V(\rho|_{M_3})=V_{2^{N-m_*}}.$ If $m_3=m_4-1$ then $\V(\rho|_{M_3})=V_{2^{N-m_4}}.$

If $m_4=N$ then $\V(\rho|_{M_3})=V_{2^{N-m_3}}.$
\end{lemma}

\begin{proof}
Assume $m_4 \neq N.$ If $m_3 \neq m_4-1$ the maximum value of $k$ such that $Q = 0 \mod 2^k$ is $\min\{m_3,m_4-1\}.$ If $m_3=m_4-1$ then, since both terms in $Q$ are of the same $2$-adic valuation, the maximal value of $k$ is at least $m_4.$ However, since $\V(\rho|_{M_2})=V_{2^{s(\lambda_4)}},$ by Corollary \ref{inclusionV} it follows that $\V(\rho|_{M_3})=V_{2^{s(\lambda_4)}}.$

Now let $m_4=N.$ Then,
\begin{align}\label{Case6b}
&\text{log}_{\lambda_*}(\lambda_{2, 2^{N-k}+1})-\text{log}_{\lambda_*}(\lambda_{2,1})\\
        ={}&\alpha_3 2^{m_3} 2^{N-k} = 0 \mod 2^N\nonumber
\end{align}
when $k \leq m_3.$ Thus $k$ is maximal when $k=m_3$ and $\V(\rho|_{M_3})=V_{2^{s(\lambda_3)}}$ when $m_4=N.$
\end{proof}

We now count the number of twist isoclasses. To do this we break the computation into four subcases. Note that, in all subcases, there are $(1-2^{-1})2^N$ choices for $\lambda_2.$ For the first three subcases we assume that $s(\lambda_4)\neq 0.$

\textbf{Case 6.1} For some $M$ such that $2 \leq M \leq N-2$, let $s(\lambda_3)=M > s(\lambda_4)+1.$ We have that there are $(1-2^{-1})2^M$ choices for $\lambda_3$ and $2^{M-2}-1$ choices for $\lambda_4.$ Since  $s(\lambda_3)=M > s(\lambda_4)+1$ by Lemmas \ref{shout} and \ref{shoutM4} we must divide by $2^{M}$ to take standardization into account. Thus, in this subcase there are
\begin{align}
    &(1-2^{-1})2^N \sum_{M=2}^{N-2} (1-2^{-1})2^M(2^{M-2}-1)2^{-M}\\
 ={}&(1-2^{-1})2^N \big((2^{N-4}-1)-(1-2^{-1})(N-4)\big)\nonumber
\end{align}
twist isoclasses. Note that when $M=2$ we have that $(2^{M-2}-1)=0.$

\textbf{Case 6.2} For some $M$ such that $1 \leq M \leq N-2,$ let $s(\lambda_4)=M$ and $s(\lambda_3) < s(\lambda_4)+1=M+1.$  There are $(1-2^{-1})2^M$ choices for $\lambda_4$ and $2^M$ choices for $\lambda_3.$ By Lemmas \ref{shout} and \ref{shoutM4} we must divide by $2^{M+1}$ to take standardization into account. Thus in this subcase there are
\begin{align}
    &(1-2^{-1})2^N \sum_{M=1}^{N-2} (1-2^{-1})2^M2^{M}2^{-(M+1)}\\
 ={}&(1-2^{-1})2^N (2^{N-3}-2^{-1})\nonumber
\end{align}
twist isoclasses.

\textbf{Case 6.3} For some $M$ such that $2 \leq M \leq N-2$, let $s(\lambda_3)=M = s(\lambda_4)+1.$  We have that there are $(1-2^{-1})2^M$ choices for $\lambda_3$ and $(1-2^{-1})2^{M-1}$ choices for $\lambda_4.$ Since  $s(\lambda_3)=M = s(\lambda_4)+1$ by Lemmas \ref{shout} and \ref{shoutM4} we must divide by $2^{M-1}$ to take standardization into account. Thus, in this subcase there are
\begin{align}
    &(1-2^{-1})2^N \sum_{M=2}^{N-2} (1-2^{-1})^2 2^M2^{M-1}2^{-(M-1)}\\
 ={}&(1-2^{-1})^2 2^N (2^{N-2}-2)\nonumber
\end{align}
twist isoclasses.

\textbf{Case 6.4} Assume $s(\lambda_4)=0.$ Let $s(\lambda_3)=M$ for $0 \leq M \leq N-2.$ If $M>0$ there are $(1-2^{-1})2^M$ choices for $\lambda_3$ and there is $1$ choice for $\lambda_3$ if $M=0.$ There is only $1$ choice for $\lambda_4.$ By Lemmas \ref{shout} and \ref{shoutM4} we must divide by $2^M$ to take standardization into account. Thus in this subcase there are
\begin{align}
&(1-2^{-1})2^N\left(1+\sum_{M=1}^{N-2}(1-2^{-1})2^M2^{-M}\right)\\
={}&(1-2^{-1})2^N\big(1+(1-2^{-1})(N-2)\big)\nonumber
\end{align}
twist isoclasses.

This ends the subcase distinctions.\\

Thus, summing together all subcases we have that there are
\begin{align}\label{m4case6}
   &(1-2^{-1})2^N\big((2^{N-4}-1)-(1-2^{-1})(N-4)\big)+(1-2^{-1})2^N(2^{N-3}-2^{-1})\\
   &+(1-2^{-1})^2 2^N(2^{N-2}-2)+(1-2^{-1})2^N\big(1+(1-2^{-1})(N-2)\big) \nonumber \\
={}& (1-2^{-1})2^N(2^{N-2}+2^{N-4}-2^{-1})\nonumber
\end{align}
twist isoclasses in Case 6 when $N \geq 4.$ When $N=3$ we sum together Cases 6.2 and 6.4. Thus there are
\begin{align}
   &(1-2^{-1})2^3\left( (1-2^{-1})+1+(1-2^{-1}) \right)\\
={}& 8 \nonumber
\end{align}
twist isoclasses in Case 6. When $N=2$ we only include Case 6.4 and thus there are
\begin{equation}
(1-2^{-1})2^2(1)=2
\end{equation}
twist isoclasses in Case 6.

\textbf{Cases 7 and 8}: We note for both cases $s(\lambda_2)\leq N-1$ and $s(\lambda_4)\leq N-2.$  By Case 2.2 of Section \ref{sub:Mp+1},  we have that $\rho|_{M_3}$ has $V_{2^{N-1}}$ as a proper stable subspace, as with the previous two cases. We check whether $V_{2^{N-1}}$ is a stable subspace of $\langle y,x_1 \rangle.$
As usual, we choose a $\lambda_* \in S_2^{N}/S_2^{N-1}$ and write each $\lambda_i$ as a power of $\lambda_*$; that is $\lambda_i = \lambda_*^{\alpha_i 2^{m_i}}$ such that $2 \nmid \alpha_i,$ $m_2 \geq 1,$ $m_3 \geq 1,$  $m_4 \geq 2,$ and $i \in \{2,3,4\}.$

Appealing to Corollary \ref{minimalpower}, consider $\lambda_{1,2^{N-1}+1}$, noting that $2^{2N-2} = 0 \mod 2^N$:
\begin{align}
\Lambda :={}& \text{log}_{\lambda_*} (\lambda_{1,2^{N-1}+1})\\
 ={}& \alpha_2 2^{m_2} 2^{N-1} + \alpha_3 2^{m_3-1} 2^{N-1}(2^{N-1}+1)\nonumber\\
 &+ \alpha_4 2^{m_4-1} \frac{2^{N-1}(2^{N-1}+1)(2^{N-1}+2)}{3} \mod 2^N\nonumber\\
 ={}& \text{log}_{\lambda_*}(\lambda_{1}) + 2^{N-1}[\alpha_2 2^{m_2}+\alpha_3 2^{m_3-1}] \mod 2^N.\nonumber
\end{align}
Clearly if $m_3 \geq 2$ then the expression in the square brackets above, say $Q$, is $0 \mod 2$ and $V_{2^{N-1}}$ is indeed a stable subspace of $\rho.$ If $m_3=1$ then $Q$ is not $0 \mod 2$ and $V_{2^{N-1}}$ is not a stable subspace of $\rho.$ Therefore $\rho$ is irreducible. So we have that in Case 8 there are no twist isoclasses. In Case 7 there are $2^{N-2}$ choices for $\lambda_4$, $(1-2^{-1})2^{N-1}$ choices for $\lambda_3,$ and $2^{N-1}$ choices for $\lambda_2.$

We now determine the behaviour of standardization in this case. It is easy to see that $\V(\rho|_{M_2})=V_{2^{s(\lambda_4)}}$ and thus $\V(\rho|_{M_3})$ is no smaller than $V_{2^{s(\lambda_4)}}$.

Let $N \geq 3;$ we calculate the case when $N=2$ separately later in the section. We write $\lambda_3,\lambda_4$ in terms of some $\lambda_* \in S_2^N \backslash S_2^{N-1}$ in the usual way, with $m_3=1$ and $m_4$ such that $2 \leq m_4 \leq N$. If $m_4=N$ then it is easy to show that $\V(\rho|_{M_3})=V_{2^{N-1}}$. Now assume $m_4 \neq N.$ As in Case 6, we use the power of Corollary \ref{minimalpower}. Consider $\Lambda:={}\log_{\lambda_*}(\lambda_{2,2^{N-k}+1})-\text{log}_{\lambda_*}(\lambda_{2,1})$ for $k$ such that $1 \leq k \leq m_4.$ Then

\begin{align}\label{wilco}
\Lambda ={}& \alpha_3 2 \cdot 2^{N-k} + \alpha_4 2^{m_4-1} 2^{N-k}(2^{N-k}+1) \mod 2^N\\
        ={}& 2^{N-k}[\alpha_3 2 + \alpha_4 2^{m_4-1}(2^{N-k}+1)]\nonumber
\end{align}
Let $Q$ be the terms in the square brackets above. We have that $Q = 0 \mod 2^k$ if and only if $\lambda_{2,1}=\lambda_{2,2^{N-k}+1}$ and thus $V_{2^{N-k}}$ is a proper stable subspace of $\langle y,x_2 \rangle.$ We break this computation into two subcases.

\textbf{Case 7.1} Assume $m_4 >2$.\\
It is clear that if $m_4 >2$ then, since $m_3=1$, by Equation \ref{wilco} the maximal $k$ such that $\Lambda= 0 \mod 2^N$ is when $k=1.$ Thus $\V(\rho|_{M_3})=V_{2^{N-1}}$. Note that $V_{2^{N-1}}$ is also minimal when $m_4=N.$
Let $s(\lambda_4)=M$ where $M \leq N-3.$ In this subcase there are $2^{N-1}$ choices for $\lambda_2,$ $(1-2^{-1})2^{N-1}$ choices for $\lambda_3$ and $2^{N-3}$ choices for $\lambda_4.$ By Lemma \ref{shout} we must divide by $2^{N-1}$ to take standardization into account. Thus, in this subcase, there are
\begin{equation}
2^{N-1}(1-2^{-1})2^{N-1}2^{N-3}2^{-(N-1)}=(1-2^{-1})2^{2N-4}
\end{equation}
twist isoclasses.

\textbf{Case 7.2} Assume $m_4=2$\\
If $m_4=2$ then we have that $Q=0 \mod 2^2$ and, since $\V(\rho|_{M_2})=V_{2^{N-2}},$ then by Corollary \ref{inclusionV} $\V(\rho|_{M_3})=V_{2^{N-2}}$. There are $2^{N-1}$ choices for $\lambda_2,$ $(1-2^{-1})2^{N-1}$ choices for $\lambda_3,$ and $(1-2^{-1})2^{N-2}$ choices for $\lambda_4.$ By Lemma \ref{shout} we must divide by $2^{N-2}$ to take standardization into account. Thus, in this subcase there are
\begin{equation}
(1-2^{-1})^22^{N-1}2^{N-1}2^{N-2}2^{-(N-2)}=(1-2^{-1})^2 2^{2N-2}
\end{equation}
twist isoclasses.

This ends the subcase distinctions.\\

Summing together these two subcases there are, for $N \geq 3,$
\begin{align}\label{m4case7}
(1-2^{-1})2^{2N-4}+(1-2^{-1})^2 2^{2N-2}=(1-2^{-1})2^{2N-4}(1+2^2(1-2^{-1}))
\end{align}
twist isoclasses.

Now assume $N=2.$ Then we have that $s(\lambda_4)=0.$ A short calculation shows that $\lambda_{2,1}=\lambda_{2,3}$ and by Corollary \ref{minimalpower} we have that $V_2$ is a minimal stable subspace. By Lemma \ref{shout} we must divide by $2$ to take standardization into account. There are $2^1$ choices for $\lambda_2,$ $(1-2^{-1})2^{1}=1$ choice for $\lambda_3,$ and $1$ choice for $\lambda_4.$ Thus, in this subcase there is
\begin{equation}
2\cdot1\cdot1\cdot2^{-1}=1
\end{equation}
twist isoclass.\\

\noindent This ends the case distinctions.\\

We now consider the case when $N=1.$ Note that, for clarity, we will call $\iota$ the square root of $-1.$ By Equation \ref{c} we have that $s(\lambda_4)\leq 1;$ that is, $\lambda_4 \in \{1,-1\}.$
If $\lambda_4=-1,$ then by Equation \ref{cc} we have that $\lambda_3 \in \{\iota, -\iota\}$ and by Equation \ref{ccc} we have that $\lambda_2 \in \{ \pm \sqrt{\iota}, \pm \sqrt{\iota}\}$ such that $\lambda_2^2 = -\lambda_3.$

If $\lambda_4=1$ then by Equation \ref{cc} we have that $\lambda_3 \in \{1,-1\}.$ If $\lambda_3=1$ then by Equation \ref{ccc} and, since $\rho$ is not the identity representation, we have that $\lambda_2=-1.$ If $\lambda_3=-1$ then by Equation \ref{ccc} we have that $\lambda_2 \in \{\iota,-\iota\}.$

A set of choices of the $\lambda_i$ gives us an irreducible representation if and only if $\lambda_{i,1} \neq \lambda_{i,2}$ holds for at least one $ 1 \leq i \leq 3;$ that is, one of the following is true:
\begin{align}\label{cN1}
\lambda_4 \neq 1\\
\lambda_3\lambda_4 \neq 1\\
\lambda_2\lambda_3\lambda_4 \neq 1.
\end{align}
It is easy to see that all of our choices of sets of $\lambda_i$ give us irreducible representations. For triples $(\lambda_4,\lambda_3,\lambda_2)$ it is easy to check that the pairs \linebreak
\scalebox{0.97}{$\left[(-1, \iota, \sqrt{\iota}),(-1,-\iota,-\sqrt{-\iota})\right], \left[(-1, \iota, -\sqrt{\iota}),(-1,-\iota,\sqrt{-\iota})\right], \left[(1,-1,\iota),(1,-1,-\iota)\right]$} \linebreak
 are standardization equivalent. Therefore we can say that
\begin{equation}\label{m4N1}
r_2(M_4)=4.
\end{equation}

We count the number of twist isoclasses for $N=2,3$ separately as well. Summing Cases 1 through 8 for $N=2,3$ we have $r_4(M_4)= 17$ and
$r_8(M_4)=70.$

We can now compute the $2$-local representation growth zeta function of $M_4$ by summing together the number of twist isoclasses from Cases 1 through 8:
\begin{align}
\zeta^{irr}_{M_{4},2}(s) ={}& 1 + 4\cdot 2^{-s}+ 17\cdot2^{-2s}+70\cdot2^{-3s} \\
&+\sum_{N=4}^{\infty} \left[(1-2^{-1})^42^{2N+3}+(1-2^{-1})^3 2^{2N}\right.\nonumber \\
&+ (1-2^{-1})2^{2N-2}+(1-2^{-1})^2 2^{2N-1} \nonumber\\
&+ (1-2^{-1})2^N(2^{N-2}+2^{N-4}-2^{-1})\nonumber\\
&+ \left.(1-2^{-1})2^{2N-4}(1+2^2(1-2^{-1}))\right]2^{-Ns}\nonumber.
\end{align}
\noindent Simplifying the expression above, we obtain
\begin{equation}
\zeta^{irr}_{M_{4},2}(s) = \frac{(1-2^{-s})^2}{(1-2^{1-s})(1-2^{2-s})}.
\end{equation}
Note that this result is the same as the zeta function for non-exceptional primes in \cite{MaxClassI}. It is then easy to check that this does satisfy the functional equation in \cite{Voll1}.

Now that we have the $p$-local representation zeta functions of $M_4$ we can now state the global representation zeta function:

\begin{equation}
\zeta^{irr}_{M_4}(s)= \dfrac{\zeta(s-1)\zeta(s-2)}{(\zeta(s))^2}.
\end{equation} 

\bibliographystyle{acm}
\bibliography{ThesisReferences}{}

\end{document}